\documentclass[11pt,reqno]{amsart}

\usepackage{amsmath,amssymb,amsthm,bbm,bm}
\usepackage{color}
\usepackage[colorlinks=true, allcolors=blue]{hyperref}
\usepackage{mathrsfs}
\usepackage{mathtools}

\usepackage[noabbrev,capitalize,nameinlink]{cleveref}
\crefname{equation}{}{}
\usepackage[noadjust]{cite}

\usepackage{fullpage}

\usepackage{graphics}
\usepackage{pifont}
\usepackage[T1]{fontenc}
\usepackage{tikz}
\usetikzlibrary{arrows.meta}
\usepackage{environ}
\usepackage{framed}
\usepackage{url}
\usepackage[linesnumbered,ruled,vlined]{algorithm2e}
\usepackage[noend]{algpseudocode}
\usepackage[labelfont=bf]{caption}
\usepackage{cite}
\usepackage{framed}
\usepackage[framemethod=tikz]{mdframed}
\usepackage{appendix}
\usepackage{graphicx}
\usepackage[textsize=tiny]{todonotes}
\usepackage{tcolorbox}
\usepackage{enumerate}
\allowdisplaybreaks[1]
\usepackage{enumerate}

\crefname{algocf}{Algorithm}{Algorithms}

\crefname{equation}{}{} 
\AtBeginEnvironment{appendices}{\crefalias{section}{appendix}} 

\usepackage[final]{showkeys} 
\colorlet{refkey}{orange!20}
\colorlet{labelkey}{blue!30}

\crefname{algocf}{Algorithm}{Algorithms}

\numberwithin{equation}{section}
\newtheorem{theorem}{Theorem}[section]
\newtheorem{proposition}[theorem]{Proposition}
\newtheorem{lemma}[theorem]{Lemma}

\crefname{claim}{Claim}{Claims}

\newtheorem*{question*}{Question}

\theoremstyle{definition}
\newtheorem{definition}[theorem]{Definition}

\newtheorem*{definition*}{Definition}

\theoremstyle{remark}

\let\originalleft\left
\let\originalright\right
\renewcommand{\left}{\mathopen{}\mathclose\bgroup\originalleft}
\renewcommand{\right}{\aftergroup\egroup\originalright}
\usepackage{pgfplots}
\usetikzlibrary{pgfplots.groupplots}
\usepackage{verbatim}

\newcommand{\mb}{\mathbb}

\newcommand{\on}{\operatorname}

\global\long\def\RR{\mathbb{R}}%

\global\long\def\E{\mathbb{E}}%

\global\long\def\NN{\mathbb{N}}%

\global\long\def\ZZ{\mathbb{Z}}%

\global\long\def\supp{\operatorname{supp}}%

\global\long\def\range#1{\left[#1\right]}%

\global\long\def\d{\operatorname{d}\!}%

\global\long\def\floor#1{\left\lfloor #1\right\rfloor }%

\global\long\def\ceil#1{\left\lceil #1\right\rceil }%

\global\long\def\cond{\,\middle|\,}%

\let\polishL\L

\global\long\def\L{\mathcal{L}}%

\DeclareRobustCommand{\L}{\ifmmode{\mathcal{L}}\else\polishL\fi}

\global\long\def\ord{\mathcal{O}}%

\global\long\def\ordm#1{\mathcal{O}_{#1}}%

\global\long\def\cordm#1#2{\mathcal{O}_{#2}^{#1}}%

\global\long\def\cSm#1#2{\mathcal{L}_{#2}^{#1}}%

\global\long\def\ext#1{\mathcal{O}^\ast\left(#1\right)}%

\global\long\def\randN{\boldsymbol{N}}%

\global\long\def\randX{\boldsymbol{X}}%

\global\long\def\rg#1#2{\RR\left(#1,#2\right)}%

\global\long\def\randG{\boldsymbol{G}}%

\global\long\def\randY{\boldsymbol{Y}}%

\global\long\def\randZ{\boldsymbol{Z}}%

\global\long\def\randM{\boldsymbol{M}}%

\global\long\def\randL{\boldsymbol{L}}%

\allowdisplaybreaks

\title{Note on random Latin squares and the triangle removal process}
\author[Kwan]{Matthew Kwan}

\address{Institute of Science and Technology (IST) Austria}

\email{matthew.kwan@ist.ac.at}

\author[A2]{Ashwin Sah}
\author[A3]{Mehtaab Sawhney}
\address{Department of Mathematics, Massachusetts Institute of Technology, Cambridge, MA}
\email{\{asah,msawhney\}@mit.edu}

\thanks{Part of this research was done while Kwan was working at Stanford University and ETH Zurich, and was supported in part by NSF grant DMS-1953990 and SNSF project 178493. Sah was supported by The Paul \& Daisy Soros Fellowship. Sah and Sawhney were supported by NSF Graduate Research Fellowship Program DGE-1745302.}

\begin{document}

\begin{abstract}
This is a companion note to the paper ``Almost all Steiner triple
systems have perfect matchings'' (arXiv:1611.02246). That paper contains
several general lemmas about random Steiner triple systems; in this
note we record analogues of these lemmas for random Latin squares,
which in particular are necessary ingredients for our recent paper
``Large deviations in random Latin squares'' (arXiv:2106.11932). Most important is a relationship between uniformly random order-$n$ Latin squares and the triangle removal process on the complete tripartite graph $K_{n,n,n}$.
\end{abstract}

\maketitle

\section{Introduction}

An order-$n$ \emph{Latin square} is usually defined as an $n\times n$ array
of the numbers between 1 and $n$ (we call these \emph{symbols}),
such that each row and column contains each symbol exactly once. In \cite{Kwa20}, Kwan introduced some general probabilistic techniques for studying so-called \emph{Steiner triple systems}, and described how these techniques can be extended to Latin squares. The purpose of this note is to record complete proofs of various lemmas about random Latin squares, which are analogues of the lemmas in \cite{Kwa20}. In particular, these lemmas are ingredients in our recent paper on large deviations on random Latin squares~\cite{KSS21}.

We emphasise that the proofs in this note are almost exactly the same as the proofs of corresponding lemmas in \cite{Kwa20}; the goal of this note is completeness, not new ideas. Also, we refer the reader to \cite{Kwa20,KSS21} for further references, motivation and background on this topic.

First, it will be more convenient for us to make a slightly different (equivalent)
definition of a Latin square, in terms of 3-uniform hypergraphs.
\begin{definition}[Latin squares]Define
\[
R=R_{n}=\left\{ 1,\dots,n\right\} ,\qquad C=C_{n}=\left\{ n+1,\dots,2n\right\} ,\qquad S=S_{n}=\left\{ 2n+1,\dots,3n\right\} .
\]
We call the elements of $R,C,S$ \emph{rows}, \emph{columns} and \emph{symbols}
respectively. Then, a \emph{partial Latin square} (of order $n$)
is a 3-partite 3-uniform hypergraph with 3-partition $V:=R\cup C\cup S$,
such that no pair of vertices is involved in more than one edge. Let
$\L_{m}$ be the set of partial Latin squares with $m$ hyperedges.
A \emph{Latin square} is a partial Latin square with exactly $N:=n^{2}$
hyperedges (this is the maximum possible, and implies that every pair
of vertices in different parts is contained in exactly one edge).
Let $\L$ be the set of Latin squares.
\end{definition}

\begin{definition}[Ordered Latin squares]
Let $\ord$ be the set of ordered Latin squares (i.e., Latin squares
with an ordering on their set of hyperedges), and let $\ordm m$ be
the set of ordered partial Latin squares with $m$ hyperedges. For
$L\in\ordm m$ and $i\le m$, let $L_{i}$ be the ordered partial
Latin square consisting of just the first $i$ hyperedges of $L$.
\end{definition}

\begin{definition}[Triangle removal process]
The (3-partite) \emph{triangle removal process} is defined as follows.
Start with the complete 3-partite graph $K_{n,n,n}$ on the vertex
set $R\cup C\cup S$. At each step, consider the set of all triangles
in the current graph, select one uniformly at random, and remove it.
Note that after $m$ steps of this process, the removed triangles
can be interpreted as an ordered partial Latin square $L\in\ord_{m}$
(unless we run out of triangles before the $m$th step). Let $\RR(n,m)$
be the distribution on $\ord_{m}\cup\{\ast\}$ obtained from $m$ steps
of the triangle removal process (where ``$\ast$'' corresponds to the
event that we run out of triangles). Note that it also makes sense to run the triangle removal process starting from some $G\subseteq K_{n,n,n}$ instead of starting from $K_{n,n,n}$ itself.
\end{definition}

\begin{definition}[Quasirandomness]
For this definition (and occasionally henceforth) we write $V^{1},V^{2},V^{3}$ instead of $R,C,S$
for the three parts of $K_{n,n,n}$, and we write $V=V^{1}\cup V^{2}\cup V^{3}$
for the vertex set of $K_{n,n,n}$. The \emph{density} of a subgraph $G\subseteq K_{n,n,n}$
is defined to be $d\left(G\right)=e\left(G\right)/\left(3N\right)$. A subgraph
$G\subseteq K_{n,n,n}$ is \emph{$(\varepsilon,h)$-quasirandom} if
for each $q\in\{1,2,3\}$, every set $A\subseteq V\setminus V^{q}$
with $|A|\le h$ has $(1\pm\varepsilon)d(G)^{|A|}n$ common neighbours
in $V^{q}$. For a (possibly ordered) partial Latin square $L$, let
$G(L)$ be the graph consisting of those edges of $K_{n,n,n}$ which
are not included in any hyperedge of $L$ (so if $m=N$ ($=n^{2}$)
then $G(L)$ is always the empty graph, and if $m=0$ then always
$G(L)=K_{n,n,n}$). Let $\mathcal{L}_{m}^{\varepsilon,h}$ be the
set of partial Latin squares $P\in\mathcal{L}_{m}$ such that $G(P)$
is $(\varepsilon,h)$-quasirandom, and let $\cordm{\varepsilon,h}m\subseteq\ordm m$
be the set of ordered partial Latin squares $L\in\ordm m$ such that
$L_{i}\in\cSm{\varepsilon,h}i$ for each $i\le m$.
\end{definition}

\begin{definition}[Binomial random hypergraph]
Let $\mathbb{G}^{3}\left(n,p\right)$ be the probability distribution
on 3-partite 3-uniform hypergraphs with vertex set $R\cup C\cup S$,
where every possible hyperedge respecting the 3-partition is included
with probability $p$ (so, the expected number of edges is $pn^{3}$).
\end{definition}

Now, our lemmas are as follows. Recall that $N=n^2$. The first lemma states that quasirandom partial Latin squares have similar amounts of completions, up to multiplicative factors of $\exp(O(n^{2-\Omega(1)}))$. It is proved in \cref{sec:completion-count}.

\begin{lemma}\label{lem:num-extensions}
For an ordered partial Latin square $L\in\ordm m$,
let $\ext L\subseteq\ord$ be the set of ordered Latin squares $L^\ast$
such that $L_{m}^\ast=L$. For sufficiently large $h\in\NN$ and
any $a>0$, there is $b=b\left(a,h\right)>0$ such that the following
holds. For any fixed $\alpha\in\left(0,1\right)$, if $\varepsilon= n^{-a}$
then any $L,L'\in\cordm{\varepsilon,h}{\alpha N}$ satisfy
\[
\frac{\left|\ext L\right|}{\left|\ext{L'}\right|}\le\exp\left(O\left(n^{2-b}\right)\right).
\]
\end{lemma}

The second lemma states that quasirandom partial Latin squares are output by the triangle removal process with comparable probabilities. It is proved in \cref{sec:greedy-random-approx-uniform}.

\begin{lemma}
\label{lem:greedy-random-approx-uniform}The following holds for any
fixed $a\in\left(0,2\right)$ and $\alpha\in\left(0,1\right)$. Let
$\varepsilon=n^{-a}$, let $L,L'\in\cordm{\varepsilon,2}{\alpha N}$
and let $\randL\sim\rg n{\alpha N}$. Then
\[
\frac{\Pr\left(\randL=L\right)}{\Pr\left(\randL=L'\right)}\le\exp\left(O\left(n^{2-a}\right)\right).
\]
\end{lemma}

The third lemma essentially states that for any Latin square, most random subsets of its edges look quasirandom. It is proved in \cref{sec:randomly-ordered}.

\begin{lemma}
\label{lem:random-is-typical}The following holds for any fixed $h\in\NN$,
$\alpha\in\left(0,1\right)$ and $a\in\left(0,1/2\right)$. Let $\varepsilon=n^{-a}$,
consider any Latin square $L$, and uniformly at random order its
hyperedges to obtain an ordered Latin square $\randL\in\ord$. Then
$\Pr\left(\randL_{\alpha N}\notin\cordm{\varepsilon,h}{\alpha N}\right)=\exp\left(-\Omega\left(n^{1-2a}\right)\right)$.
\end{lemma}

The next lemma shows how to compare the triangle removal process to a nicer independent model (with deletions). It is proved in \cref{sec:coupling}.

\begin{lemma}
\label{lem:bite-transfer}Let $\mathcal{P}$ be a property of unordered
partial Latin squares that is monotone increasing in the sense that
$L\in\mathcal{P}$ and $L'\supseteq L$ implies $L'\in\mathcal{P}$.
Fix $\alpha\in\left(0,1\right)$, let $\randL\sim\RR\left(n,\alpha N\right)$,
let $\randG\sim\mathbb{G}^{3}\left(n,p\right)$ for $p = \alpha/n$ and let $\randL^\ast$
be the partial Latin square obtained from $\randG$ by deleting (all
at once) every hyperedge which intersects another hyperedge in more than one
vertex. Then
\[
\Pr\left(\randL\notin\mathcal{P}\emph{ and }\randL\neq\ast\right)=O\left(\Pr\left(\randL^\ast\notin\mathcal{P}\right)\right).
\]
\end{lemma}

As a final lemma, we verify that the triangle removal process succeeds, i.e., produces a partial Latin square instead of $\ast$, with probability $1-o(1)$ (and in fact produces a quasirandom output). It is proved in \cref{sec:random-triangle-removal}, based on \cref{thm:triangle-removal-analysis}, which is a general analysis of the triangle removal process also needed in \cref{sec:completion-count}.

\begin{lemma}\label{lem:TRP-quasirandom}
The following holds for any $h\in\mb N$. There is a constant $a=a(h)>0$ such that if $\alpha\in(0,1)$ and $\varepsilon=n^{-a}$, then for $\randL\sim\mathbb{R}(n,\alpha N)$ we have
\[\Pr(\randL\notin\mathcal{O}_{\alpha N}^{\varepsilon,h}\emph{ or }\randL=\ast)=o(1).\]
\end{lemma}

\subsection{Notation}

We use standard asymptotic notation throughout. Here and for the rest of the paper, asymptotics are as $n\to \infty$. For functions $f=f\left(n\right)$
and $g=g\left(n\right)$:
\begin{itemize}
\item $f=O\left(g\right)$ means there is a constant $C$ such that $\left|f\right|\le C\left|g\right|$,
\item $f=\Omega\left(g\right)$ means there is a constant $c>0$ such that
$f\ge c\left|g\right|$,
\item $f=\Theta\left(g\right)$ means that $f=O\left(g\right)$ and $f=\Omega\left(g\right)$,
\item $f=o\left(g\right)$ means that $f/g\to0$.
\item By ``asymptotically almost surely'', or ``a.a.s.'', we mean that the probability of an event is $1-o(1)$. In particular, to say that a.a.s.~$f=o\left(g\right)$ means that for any $\varepsilon>0$,
a.a.s.~$f/g<\varepsilon$.

\end{itemize}
Also, following \cite{Kee18c}, the notation $f=1\pm\varepsilon$ means
$1-\varepsilon\le f\le1+\varepsilon$.

We also use standard graph theory notation: $V\left(G\right)$ and
$E\left(G\right)$ are the sets of vertices and (hyper)edges of a
(hyper)graph $G$, and $v\left(G\right)$ and $e\left(G\right)$ are
the cardinalities of these sets. The subgraph of $G$ induced by a
vertex subset $U$ is denoted $G\left[U\right]$, the degree of a
vertex $v$ is denoted $\deg_{G}\left(v\right)$, and the subgraph
obtained by deleting $v$ is denoted $G-v$.

For a positive integer $n$, we write $\range n$ for the set $\left\{ 1,2,\dots,n\right\} $.
For a real number $x$, the floor and ceiling functions are denoted
$\floor x=\max\left\{ i\in\ZZ:i\le x\right\} $ and $\ceil x=\min\left\{ i\in\ZZ:i\ge x\right\} $.
We will however mostly omit floor and ceiling signs and assume large
numbers are integers, wherever divisibility considerations are not
important. All logarithms are in base $e$.

Finally, we remark that throughout the paper we adopt the convention
that random variables (and random objects more generally) are printed
in bold.

\section{Randomly ordered Latin squares}\label{sec:randomly-ordered}

In this section we prove \cref{lem:random-is-typical}.
\begin{proof}[Proof of \cref{lem:random-is-typical}]
Recall that $N=n^2$ and consider $m\le\alpha N$. Note that $\randL_{m}$ (as an unordered
partial Latin square) is a uniformly random subset of $m$ hyperedges
of $L$. Also note that 
\[
d\left(G\left(\randL_{m}\right)\right)=\frac{3N-3m}{3N}=1-\frac{m}{N}.
\]
We can obtain a random partial Latin square almost equivalent to $\randL_{m}$
by including each hyperedge of $L$ with independent probability $m/N$.
Let $\randL'$ denote the partial Latin square so obtained, and let
$\randG'=G\left(\randL'\right)$. Now, fix $q\in\left\{ 1,2,3\right\} $
and fix a set $A$ of at most $h$ vertices not in $V^q$. It suffices
to prove
\begin{align}
\left|\bigcap_{w\in A}N_{q}\left(w\right)\right| & =\left(1\pm n^{-a}\right)\left(1-\frac{m}{N}\right)^{\left|A\right|}n,\label{eq:degree-random-binomial}
\end{align}
with probability $1-\exp\left(-\Omega\left(n^{1-2a}\right)\right)$,
where $N_{q}\left(w\right)$ is the neighbourhood of $w$ in $V_{q}$,
in the graph $\randG'$. Indeed, the so-called Pittel inequality (see
\cite[p.~17]{JLR00}) would imply that the same estimate holds with
essentially the same probability if we replace $\randL'$ with $\randL_{m}$
(thereby replacing $\randG'$ with $G\left(\randL_{m}\right)$). We
would then be able to finish the proof by applying the union bound
over all $m\le\alpha N$ and all choices of $A$.

Note that there are at most ${\left|A\right| \choose 2}=O\left(1\right)$
hyperedges of $L$ that include more than one vertex in $A$ (by the
defining property of a Latin square). Let $U$ be the set of vertices
involved in these atypical hyperedges, plus the vertices in $A$,
so that $\left|U\right|=O\left(1\right)$. Let $\randN=\left|\left(\bigcap_{w\in A}N_{q}\left(w\right)\right)\backslash U\right|$.
For every $v\in V^{q}\setminus U$ and $w\in A$ there is exactly
one hyperedge $e_{v}^{w}$ in $L$ containing $v$ and $w$, whose
presence in $\randL'$ would prevent $v$ from contributing to $\randN$.
For each fixed $v\in V^q\setminus U$ the hyperedges $e_{v}^{w}$, for $w\in A$,
are distinct by definition of $U$, so
\[
\Pr\left(v\in\bigcap_{w\in A}N_{q}\left(w\right)\right)=\left(1-\frac{m}{N}\right)^{\left|A\right|},
\]
and thus by linearity of expectation $\E\randN=\left(1-m/N\right)^{\left|A\right|}\left(n-|U|\right)$.
Now, $\randN$ is determined by the presence of at most $\left(n-\left|U\right|\right)\left|A\right|=O\left(n\right)$
hyperedges in $\randL'$, and changing the presence of each affects
$\randN$ by at most $2=O\left(1\right)$. So, by the Azuma\textendash Hoeffding
inequality (see \cite[Section~2.4]{JLR00}),
\begin{align*}
\Pr\left(\left|\randN-\left(1-\frac{m}{N}\right)^{\left|A\right|}n\right|>n^{-a}\left(1-\frac{m}{N}\right)^{\left|A\right|}n-|U|\right) & \le\exp\left(-\Omega\left(\frac{\left(n^{-a}\left(1-\alpha\right)^{h}n\right)^{2}}{n}\right)\right)\\
 & =\exp\left(-\Omega\left(n^{1-2a}\right)\right).
\end{align*}
Finally, we recall that $\left|\left(\bigcap_{w\in A}N_{q}\left(w\right)\right)\right|=\randN\pm\left|U\right|$,
which completes the proof of \cref{eq:degree-random-binomial}.
\end{proof}

\section{Approximate uniformity of the triangle removal process}\label{sec:greedy-random-approx-uniform}

In this section we prove \cref{lem:greedy-random-approx-uniform}.
We first make the simple observation that the number of triangles
in a quasirandom graph $G$ can be easily estimated in terms of the
density of $G$.
\begin{proposition}
\label{prop:H-count}Let $G\subseteq K_{n,n,n}$ be an $\left(\varepsilon,2\right)$-quasirandom
graph on $n$ vertices with $\varepsilon\in(0,1]$. Then the number of triangles in $G$ is $\left(1\pm O\left(\varepsilon\right)\right)n^{3}d\left(G\right)^{3}$.
\end{proposition}

\begin{proof}
For every vertex $v\in R$, its degree in $C$ is $\left(1\pm\varepsilon\right)nd\left(G\right)$.
So, there are $\left(1\pm\varepsilon\right)n^{2}d\left(G\right)$
edges between $R$ and $C$. Then, for each such edge, the number
of ways to add a vertex $u\in S$ to create a triangle is $\left(1\pm\varepsilon\right)nd\left(G\right)^{2}$.
The desired result follows.
\end{proof}
Now we are ready to prove \cref{lem:greedy-random-approx-uniform}.
\begin{proof}[Proof of \cref{lem:greedy-random-approx-uniform}]
Each $G\left(L_{i}\right)$ has 
\[
\left(1\pm O\left(n^{-a}\right)\right)\left(1-\frac{i}{N}\right)^{3}n^{3}
\]
triangles, by $\left(n^{-a},2\right)$-quasirandomness and \cref{prop:H-count}.
We therefore have
\[
\Pr\left(\randL=L\right)=\prod_{i=0}^{\alpha N-1}\frac{1}{\left(1\pm O\left(n^{-a}\right)\right)\left(1-i/N\right)^{3}n^{3}},
\]
and a similar expression holds for $\Pr\left(\randL=L'\right)$. Taking
quotients term-by-term gives
\begin{align*}
\frac{\Pr\left(\randL=L\right)}{\Pr\left(\randL=L'\right)} & \le\left(1+O\left(n^{-a}\right)\right)^{\alpha N}\\
 & \le\exp\left(O\left(n^{2-a}\right)\right)
\end{align*}
as desired.
\end{proof}

\section{A coupling lemma}\label{sec:coupling}

In this section we prove \cref{lem:bite-transfer}.
\begin{proof}[Proof of \cref{lem:bite-transfer}]
Observe that $\randL^\ast$ can be
coupled with $\randL$ in such a way that, if $e(\randG)\le \alpha N$, then either $\randL^\ast\subseteq\randL$ or $\randL = \ast$. Indeed, an equivalent way to define the triangle removal process (and thus the distribution of $\randL$) is to take a uniformly random ordering of the triangles in $K_{n,n,n}$, go through the triangles in order, and accept each triangle if it is edge-disjoint from previously accepted triangles. Note that a random ordering of the hyperedges of $\randG$ can be viewed as the first $e(\randG) \sim \on{Bin}(n^3,\alpha/n)$ elements of a random
ordering of the set of triangles of $K_{n,n,n}$, and the triangle
removal process with this ordering produces a superset of $\randL^\ast$ whenever $e(\randG)\le\alpha N$ and $\randL\neq\ast$ (since every triangle in $\randL^\ast$ by definition does not share an edge with the prior triangles).

It follows from this and the monotonicity of $\mathcal{P}$ that
\[
\Pr\left(\randL\notin\mathcal{P}\text{ and }\randL\neq\ast\right)\le\Pr\left(\randL^\ast\notin\mathcal{P}\cond e\left(\randG\right)\le\alpha N\right).
\]
Next, since $e\left(\randG\right)$ has a binomial distribution
with mean $\E e\left(\randG\right)=n^{3}\alpha/n=\alpha N$, it
is easy to see that $\Pr\left(e\left(\randG\right)\le\alpha N\right)=\Omega\left(1\right)$.
It follows that
\begin{align*}
\Pr\left(\randL\notin\mathcal{P}\text{ and }\randL\neq\ast\right) & \le\Pr\left(\randL^\ast\notin\mathcal{P}\right)/\Pr\left(e\left(\randG\right)\le\alpha N\right)=O\left(1\right)\cdot\Pr\left(\randL^\ast\notin\mathcal{P}\right).\qedhere
\end{align*}
\end{proof}

\section{Counting completions of partial Latin squares}\label{sec:completion-count}

\global\long\def\coG{G}%

\global\long\def\Sext#1{\mathcal{L}^\ast\left(#1\right)}%

\global\long\def\randl{\boldsymbol{\lambda}}%

\global\long\def\randz{\boldsymbol{z}}%

\global\long\def\randR{\boldsymbol{R}}%

In this section we prove \cref{lem:num-extensions}. As always, recall that $N=n^{2}$.

For a partial Latin square $L\in\mathcal{L}_{\alpha N}$, let $\Sext L$
be the number of full Latin squares that include $L$. We want to determine
$\left|\ext L\right|=\left(N-\alpha N\right)!\left|\Sext L\right|$
up to a factor of $e^{n^{2-b}}$ (for some $b>0$).

First, we can get an upper bound via the \emph{entropy method}. Before we begin the proof, we briefly remind the reader of the basics
of the notion of entropy. For random elements $\randX,\randY$ with
supports $\supp\randX$, $\supp\randY$, we define the (base-$e$)
\emph{entropy }
\[
H\left(\randX\right)=-\sum_{x\in\supp\randX}\Pr\left(\randX=x\right)\log\left(\Pr\left(\randX=x\right)\right)
\]
and the\emph{ conditional entropy}
\[
H\left(\randX\cond\randY\right)=\sum_{y\in\supp\randY}\Pr\left(\randY=y\right)H\left(\randX\cond\randY=y\right).
\]
We will use two basic properties of entropy. First, we always have
$H\left(\randX\right)\le\log\,\left|\supp\randX\right|$, with equality
only when $\randX$ has the uniform distribution on its support. Second,
for any sequence of random elements $\randX_{1},\dots,\randX_{n}$,
we have
\[
H\left(\randX_{1},\dots,\randX_{n}\right)=\sum_{i=1}^{n}H\left(\randX_{i}\cond\randX_{1},\dots,\randX_{i-1}\right).
\]
See for example \cite{CT12} for an introduction to the notion of
entropy and proofs of the above two facts.
\begin{theorem}
\label{thm:number-of-completions-upper}For any $a\in(0,2)$, any $\alpha\in\left[0,1\right]$,
and any $L\in\cSm{n^{-a},2}{\alpha N}$,
\[
\left|\Sext L\right|\le\left(\left(1+O\left(n^{-a}+n^{-1/2}\right)\right)\left(\frac{1-\alpha}{e}\right)^{2}n\right)^{N\left(1-\alpha\right)}.
\]
\end{theorem}

\begin{proof}
Let $\randL^\ast\in\Sext L$ be a uniformly random completion of $L$.
We will estimate the entropy $H\left(\randL^\ast\right)=\log\,\left|\Sext L\right|$
of $\randL^\ast$.

Let $\coG=G\left(L\right)$. For each $e=\left\{ x,y\right\} \in\coG\left[R\cup C\right]$,
let $\left\{ x,y,\randz_{e}\right\} $ be the hyperedge that includes
$e$ in $\randL^\ast$ (i.e., in the $n\times n$ array formulation
of a Latin square, $\randz_{e}$ is the symbol in the cell corresponding
to $\left\{ x,y\right\} $). So, the sequence $\left(\randz_{e}\right)_{e\in\coG\left[R\cup C\right]}$
determines $\randL^\ast$. For any ordering $\prec$ on the edges of $\coG\left[R\cup C\right]$,
we have
\begin{equation}
H\left(\randL^\ast\right)=\sum_{e\in\coG\left[R\cup C\right]}H\left(\randz_{e}\cond\left(\randz_{e'}\,\colon\,e'\prec e\right)\right).\label{eq:H(S)}
\end{equation}
Now, a sequence $\lambda\in\left[0,1\right]^{R\times C}$ with all
$\lambda_{e}$ distinct induces an ordering $\prec_\lambda$ on the edges of $\coG\left[R\cup C\right]$,
with $e'\prec_\lambda e$ when $\lambda_{e'}>\lambda_{e}$. Let $\randR_{e}\left(\lambda\right)$
be an upper bound on $\left|\supp\left(\randz_{e}\cond\{\randz_{e'}\,\colon\,\lambda_{e'}>\lambda_{e}\}\right)\right|$
defined as follows for $e = \{x,y\}$. $\randR_{e}\left(\lambda\right)$ is 1 plus the
number of vertices $v\in S\setminus\left\{ \randz_{e}\right\} $
such that $\left\{ x,v\right\} ,\left\{ y,v\right\} \in\coG$, and
$\lambda_{e'}<\lambda_{e}$ for both the $e'\in\coG\left[R\cup C\right]$
included in the hyperedges that include $\left\{ x,v\right\} $ and
$\left\{ y,v\right\} $ in $\randL^\ast$. (In the $n\times n$ array formulation of a Latin square, this is just the number of symbols whose position has not yet been revealed in the row and column specified by $e$). Note that $\randR_e(\lambda)$ is random depending on all of $(\randz_e)_{e\in\coG\left[R\cup C\right]}$ (hence $\randL^\ast$), even though the support we are bounding is only a function of $(\randz_{e'})_{e'\prec_\lambda e}$.

Since $\randR_{e}\left(\lambda\right)$
is an upper bound on $\left|\supp\left(\randz_{e}\cond\randz_{e'}\,\colon\,\lambda_{e'}>\lambda_{e}\right)\right|$,
we have
\begin{equation}
H\left(\randz_{e}\cond\{\randz_{e'}\,\colon\,\lambda_{e'}>\lambda_{e}\}\right)\le\E\left[\log\randR_{e}\left(\lambda\right)\right].\label{eq:H(z)}
\end{equation}
It follows from \cref{eq:H(S)} applied to $\prec_\lambda$ and \cref{eq:H(z)} that
\[
H\left(\randL^\ast\right)\le\sum_{e\in\coG[R\cup C]}\E\left[\log\randR_{e}\left(\lambda\right)\right].
\]
This is true for any fixed $\lambda$, so it is also true if $\lambda$
is chosen randomly, as follows. Let $\randl=\left(\randl_{e}\right)_{e\in\coG}$
be a sequence of independent random variables, where each $\randl_{e}$
has the uniform distribution in $\left[0,1\right]$. (With probability
1 each $\randl_{v}$ is distinct from the others). Then
\[
H\left(\randL^\ast\right)\le\sum_{e\in\coG[R\cup C]}\E\left[\log\randR_{e}\left(\randl\right)\right].
\]
Next, for any $L^\ast\in\Sext L$ and $\lambda_{e}\in\left[0,1\right]$,
let 
\[
R_{e}^{L^\ast,\lambda_{e}}=\E\left[\randR_{e}\left(\randl\right)\cond\randL^\ast=L^\ast,\,\randl_{e}=\lambda_{e}\right].
\]
(Note that $\randl_{e}=\lambda_{e}$ occurs with probability zero,
so formally we should condition on $\randl_{e}=\lambda_{e}\pm\d\lambda_{e}$
and take limits in what follows, but there are no continuity issues
so we will ignore this detail). Now, in $\coG$, by $\left(n^{-a},2\right)$-quasirandomness,
$x$ and $y$ have $\left(1+O\left(n^{-a}\right)\right)\left(1-\alpha\right)^{2}n$
common neighbours (in $S$) other than $\randz_{e}$. By the definition of $\randR_{e}\left(\randl\right)$
and linearity of expectation, we have 
\[
R_{e}^{L^\ast,\lambda_{e}}=1+\left(1+O\left(n^{-a}\right)\right)\left(1-\alpha\right)^{2}\lambda_{e}^{2}n.
\]
By Jensen's inequality,
\[
\E\left[\log\randR_{e}\left(\randl\right)\cond\randL^\ast=L^\ast,\,\randl_{e}=\lambda_{e}\right]\le\log R_{e}^{L^\ast,\lambda_{e}}.
\]
We then have 
\begin{align*}
\E\left[\log\randR_{e}\left(\randl\right)\cond\randL^\ast=L^\ast\right] & \le\E\left[\log R_{e}^{L^\ast,\randl_{e}}\right]\\
 & =\int_{0}^{1}\log\left(1+\left(1+O\left(n^{-a}\right)\right)\left(1-\alpha\right)^{2}\lambda_{e}^{2}n\right)\d\lambda_{e}.
\end{align*}
For $C>0$ we can compute
\begin{align}
\int_{0}^{1}\log\left(1+Ct^{2}\right)\d t & =\log\left(1+C\right)-2+\frac{2\arctan\sqrt{C}}{\sqrt{C}},\label{eq:complicated-integral}
\end{align}
so (taking $C=\left(1+O\left(n^{-a}\right)\right)\left(1-\alpha\right)^{2}n$)
we deduce
\[
\E\left[\log\randR_{e}\left(\randl\right)\cond\randL^\ast=L^\ast\right]\le\log\left(\left(1-\alpha\right)^{2}n\right)-2+O\left(n^{-a}+n^{-1/2}\right).
\]
We conclude that
\begin{align*}
\log\,\left|\Sext L\right| & = H\left(\randL^\ast\right)\\
 & \le\sum_{e\in\coG[R\cup C]}\E\left[\log\randR_{e}\left(\randl\right)\right]\\
 & \le\left(N-\alpha N\right)\left(\log\left(\left(1-\alpha\right)^{2}n\right)-2+O\left(n^{-a}+n^{-1/2}\right)\right),
\end{align*}
which is equivalent to the theorem statement.
\end{proof}
For the lower bound, we will count ordered Latin squares.
\begin{theorem}
\label{thm:number-of-completions-lower}Fixing sufficiently large
$h\in\NN$ and any $a>0$, there is $b=b\left(a,h\right)>0$ such
that the following holds. For any fixed $\alpha\in\left(0,1\right)$ and
any $L\in\cordm{n^{-a},h}{\alpha N}$,
\[
\left|\ext L\right|\ge\left(\left(1-O\left(n^{-b}\right)\right)\left(\frac{1-\alpha}{e}\right)^{2}n\right)^{N\left(1-\alpha\right)}\left(N-\alpha N\right)!.
\]
\end{theorem}

To prove \cref{thm:number-of-completions-lower} we will need an analysis
of the triangle removal process (which we provide in \cref{sec:random-triangle-removal})
and the following immediate consequence of \cite[Theorem~1.5]{Kee18c}, which counts completions of Latin squares.
\begin{theorem}
\label{thm:keevash}There are $h\in\NN$, $\varepsilon_{0},a\in\left(0,1\right)$
and $n_{0},\ell\in\NN$ such that if $L\in\cSm{\varepsilon,h}m$ is
a partial Latin square with $n\ge n_{0}$, $d\left(G\left(L\right)\right)=1-m/N\ge n^{-a}$
and $\varepsilon\le\varepsilon_{0}d\left(G\right)^{\ell}$, then $L$
can be completed to a Latin square.
\end{theorem}

\global\long\def\extm#1#2{\mathcal{O}_{#1}^{\mathrm{ext}}\left(#2\right)}%

\begin{proof}[Proof of \cref{thm:number-of-completions-lower}]
Let $h\ge2$, $\ell$, $\varepsilon_{0}$ be as in \cref{thm:keevash}.
Let $c>0$ be smaller than $a\cdot b\left(a,h\right)$ in the notation
of \cref{thm:triangle-removal-analysis}, and smaller than the ``$a$''
in \cref{thm:keevash}. Let $\varepsilon=n^{-c/\ell}/\varepsilon_{0}$ and $\varepsilon' = n^{-c}$
and $M=\left(1-\varepsilon\right)N$. Let $\randL^\ast\in\ordm M\cup\left\{\ast\right\} $
be the result of running the triangle removal process on $G\left(L\right)$
to build a partial Latin square extending $L$, until there are $M$
hyperedges. Let $\ord^\ast$ be the set of $M$-hyperedge $\left(\varepsilon',h\right)$-quasirandom
ordered partial Latin squares $L^\ast\in\cordm{\varepsilon',h}M$ extending
$L$. The choice of $c$ ensures that by \cref{thm:triangle-removal-analysis}
we a.a.s.~have $\randL^\ast\in\ord^\ast$, and then by \cref{thm:keevash}
each $L^\ast\in\ord^\ast$ can be completed to an ordered Latin square.

Now, by \cref{prop:H-count} and quasirandomness coming from the output of \cref{thm:triangle-removal-analysis}, for each $L^\ast\in\ord^\ast$
the number of triangles in each $G\left(L_{i}^\ast\right)$ is 
\[
\left(1\pm O\left(n^{-c}\right)\right)\left(1-i/N\right)^{3}n^{3},
\]
so
\[
\Pr\left(\randL^\ast=L^\ast\right)\le\prod_{i=\alpha N}^{M-1}\frac{1}{\left(1-O\left(n^{-c}\right)\right)\left(1-i/N\right)^{3}n^{3}}.
\]
As discussed, using \cref{thm:triangle-removal-analysis} we have
\[
\sum_{L^\ast\in\ord^\ast}\Pr\left(\randL^\ast=L^\ast\right)=1-o\left(1\right),
\]
so 
\begin{align*}
\left|\ord^\ast\right| & \ge\left(1-o\left(1\right)\right)\prod_{i=\alpha N}^{M-1}\left(1-O\left(n^{-c}\right)\right)\left(1-\frac{i}{N}\right)^{3}n^{3}\\
 & =\left(\left(1-O\left(n^{-c}\right)\right)n^{3}\right)^{M-\alpha N}\exp\left(3\sum_{i=\alpha N}^{M-1}\log\left(1-\frac{i}{N}\right)\right).
\end{align*}
Now, note that 
\[
\sum_{i=\alpha N}^{M-1}\frac{1}{N}\log\left(1-\frac{i+1}{N}\right)\le\int_{\alpha}^{\left(1-\varepsilon\right)}\log\left(1-t\right)\d t\le\sum_{i=\alpha N}^{M-1}\frac{1}{N}\log\left(1-\frac{i}{N}\right).
\]
We compute
\begin{align*}
\sum_{i=\alpha N}^{M}\left(\log\left(1-\frac{i}{N}\right)-\log\left(1-\frac{i+1}{N}\right)\right) & =\sum_{i=\alpha N}^{M}\log\left(1+\frac{1}{N-\left(i+1\right)}\right)\\
 & \le\sum_{i=\alpha N}^{M}\frac{1}{N-\left(i+1\right)}\\
 & =O\left(\log n\right),
\end{align*}
so, noting that $\int\log s\d s=s\left(\log s-1\right)$,
\begin{align*}
3\sum_{i=\alpha N}^{M}\log\left(1-\frac{i}{N}\right) & =3N\int_{\alpha}^{\left(1-\varepsilon\right)}\log\left(1-t\right)\d t+O\left(\log n\right)\\
 & =3N\int_{\varepsilon}^{\left(1-\alpha\right)}\log s\d s+O\left(\log n\right)\\
 & =3N\left(\left(1-\alpha\right)\left(\log\left(1-\alpha\right)-1\right)-\varepsilon\left(\log\varepsilon-1\right)\right)+O\left(\log n\right),\\
\exp\left(3\sum_{i=\alpha N}^{M}\log\left(1-\frac{i}{N}\right)\right) & =\left(\left(1+O\left(n^{-c/\ell}\log n\right)\right)\frac{1-\alpha}{e}\right)^{3N\left(1-\alpha\right)}.
\end{align*}
For $b<c/\ell$, it follows from this and $M = (1-\varepsilon)N$ that
\begin{align*}
\left|\ord^\ast\right| & \ge\left(\left(1-O\left(n^{-b}\right)\right)\frac{n^{3}\left(1-\alpha\right)^{3}}{e^{3}}\right)^{\left(1-\alpha\right)N}\\
 & =\left(\left(1-O\left(n^{-b}\right)\right)\left(\frac{1-\alpha}{e}\right)^{2}n\right)^{\left(1-\alpha\right)N}\left(N-\alpha N\right)!.
\end{align*}
Recalling that each $L^\ast\in\ord^\ast$ can be completed to a full Latin square, the desired
result follows.
\end{proof}
Now, it is extremely straightforward to prove \cref{lem:num-extensions}.
\begin{proof}
Let $b\le\min\left\{ a,1/2\right\} $ and $h\ge2$ satisfy \cref{thm:number-of-completions-lower}.
By \cref{thm:number-of-completions-upper} we have 
\[
\left|\ext L\right| = \left|\Sext L\right|\left(N-\alpha N\right)!\le\left(\left(1+O\left(n^{-b}\right)\right)\left(\frac{1-\alpha}{e}\right)^{2}n\right)^{N\left(1-\alpha\right)}\left(N-\alpha N\right)!,
\]
and by \cref{thm:number-of-completions-lower} we have
\[
\left|\ext{L'}\right|\ge\left(\left(1-O\left(n^{-b}\right)\right)\left(\frac{1-\alpha}{e}\right)^{2}n\right)^{N\left(1-\alpha\right)}\left(N-\alpha N\right)!.
\]
Dividing these bounds gives
\[
\frac{\left|\ext L\right|}{\left|\ext{L'}\right|}\le\left(1+O\left(n^{-b}\right)\right)^{N\left(1-\alpha\right)}\le\exp\left(O\left(n^{2-b}\right)\right).\qedhere
\]
\end{proof}

\section{An analysis of the triangle removal process}\label{sec:random-triangle-removal}

\global\long\def\randQ{\boldsymbol{Q}}%

\global\long\def\randY{\boldsymbol{Y}}%

\global\long\def\randT{\boldsymbol{T}}%

\global\long\def\randP{\boldsymbol{P}}%

\global\long\def\randZ{\boldsymbol{Z}}%

In this section we prove \cref{thm:triangle-removal-analysis}, which was used in \cref{sec:completion-count}. We also deduce \cref{lem:TRP-quasirandom} from it, which will complete the proofs of all of our claims regarding random Latin squares.

The triangle removal process is defined
as follows. We start with a graph $G\subseteq K_{n,n,n}$ with say
$3N-3m$ edges, then iteratively delete (the edges of) a triangle
chosen uniformly at random from all triangles in the remaining graph.
Let 
\[
G=\randG\left(m\right),\randG\left(m+1\right),\dots
\]
be the sequence of random graphs generated by this process. This process
cannot continue forever, but we ``freeze'' the process instead of
aborting it: if $\randG\left(\randM\right)$ is the first graph in
the sequence with no triangles, then let $\randG\left(i\right)=\randG\left(\randM\right)$
for $i\ge\randM$.

Our objective in this section is to show that if $G$ is quasirandom
then the triangle removal process is likely to maintain quasirandomness
and unlikely to freeze until nearly all edges are gone.
\begin{theorem}\label{thm:triangle-removal-analysis}
For all $h\ge2$ and $a>0$
there are $b=b\left(a,h\right)>0$ and $\varepsilon'=\varepsilon'\left(a,h\right)>0$ such that the following holds. Let
$n^{-a}\le\varepsilon<\varepsilon'$ and suppose $G\subseteq K_{n,n,n}$ is
an $\left(\varepsilon,h\right)$-quasirandom graph with $N-3m$ edges.
Then a.a.s.~$\randM\ge\left(1-\varepsilon^{b}\right)N$ and moreover
for each $m\le i\le\left(1-\varepsilon^{b}\right)N$, the graph $\randG\left(i\right)$
is $\left(\varepsilon^{b},h\right)$-quasirandom.
\end{theorem}

Note that $K_{n,n,n}$ is $\left(O\left(1/n\right),h\right)$-quasirandom
for any fixed $h$, so in particular when we start the triangle removal
process from $G=K_{n,n,n}$ it typically runs almost to completion. This is encapsulated by \cref{lem:TRP-quasirandom}, which we deduce before moving on to the proof of \cref{thm:triangle-removal-analysis}.
\begin{proof}[Proof of \cref{lem:TRP-quasirandom}]
Apply \cref{thm:triangle-removal-analysis} to $G = K_{n,n,n}$, which is $(O(1/n),h)$-quasirandom, with its ``$a$'' set to $1$. Letting $a = b(1,h)/2$, the result immediately follows since $\alpha$ is a constant.
\end{proof}

To prove \cref{thm:triangle-removal-analysis}, it will be convenient
to use Freedman's inequality \cite[Theorem~1.6]{Fre75}, as follows.
(This was originally stated for martingales, but it also holds for
supermartingales with the same proof). Here and in what follows, we
write $\Delta X\left(i\right)$ for the one-step change $X\left(i+1\right)-X\left(i\right)$
in a variable $X$.
\begin{lemma}
\label{lem:freedman}Let $\randX\left(0\right),\randX\left(1\right),\dots$
be a supermartingale with respect to a filtration $\left(\mathcal{F}_{i}\right)$.
Suppose that $\left|\Delta\randX\left(i\right)\right|\le K$ for all
$i$, and let $V\left(i\right)=\sum_{j=0}^{i-1}\E\left[\left(\Delta\randX\left(j\right)\right)^{2}\cond\mathcal{F}_{j}\right]$.
Then for any $t,v>0$,
\[
\Pr\left(\randX\left(i\right)\ge\randX\left(0\right)+t\mbox{ and }V\left(i\right)\le v\mbox{ for some }i\right)\le\exp\left(-\frac{t^{2}}{2\left(v+Kt\right)}\right).
\]
\end{lemma}

\begin{proof}[Proof of \cref{thm:triangle-removal-analysis}]
For $q\in\left\{ 1,2,3\right\} $ and a set $A\subseteq V\setminus V^{q}$
of at most $h$ vertices, let $\randY_{A}\left(i\right)=\left|\bigcap_{w\in A}N_{q}^{\left(i\right)}\left(w\right)\right|$,
where $N_{q}^{\left(i\right)}\left(w\right)$ is the number of neighbours
of $w$ in $V^{q}$, in the graph $\randG\left(i\right)$. Let $p\left(i\right)=\left(1-i/N\right)$
and let $p^{k}\left(i\right)=\left(1-i/N\right)^{k}$, so that $p^{\left|A\right|}\left(i\right)n$
is the predicted trajectory of each $\randY_{A}\left(i\right)$.

Fix some large $C$ and small $c$ to be determined. We will choose
$b<c/\left(C+1\right)$ so that $e\left(i\right):=p\left(i\right)^{-C}\varepsilon^{c}\le\varepsilon^{b}$
for $i\le N\left(1-\varepsilon^{b}\right)$. This means that if the
conditions
\begin{align*}
\randY_{A}\left(i\right) & \le p^{\left|A\right|}\left(i\right)n\left(1+e\left(i\right)\right),\\
\randY_{A}\left(i\right) & \ge p^{\left|A\right|}\left(i\right)n\left(1-e\left(i\right)\right),
\end{align*}
are satisfied for all $A$, then $\randG\left(i\right)$ is $\left(e\left(i\right),h\right)$-quasirandom
(therefore $\left(\varepsilon^{b},h\right)$-quasirandom).

Let $\randT'$ be the smallest index $i\ge m$ such that for some
$A$, the above equations are violated (let $\randT'=\infty$ if this
never happens). Let $\randT=\randT'\land N\left(1-\varepsilon^{b}\right)$, where $\land$ denotes the minimum.
Define the stopped processes 
\begin{align*}
\randY_{A}^{+}\left(i\right) & =\randY_{A}\left(i\land\randT\right)-p^{\left|A\right|}\left(i\land\randT\right)n\left(1+e\left(i\land\randT\right)\right),\\
\randY_{A}^{-}\left(i\right) & =-\randY_{A}\left(i\land\randT\right)+p^{\left|A\right|}\left(i\land\randT\right)n\left(1-e\left(i\land\randT\right)\right).
\end{align*}
We want to show that for each $A$ and each $s\in\left\{ +,-\right\} $,
the process $\randY_{A}^{s}=\left(\randY_{A}^{s}\left(i\right),\randY_{A}^{s}\left(i+1\right),\dots\right)$
is a supermartingale, and then we want to use \cref{lem:freedman}
and the union bound to show that a.a.s.~each $\randY_{A}^{s}$ only
takes negative values.

To see that this suffices to prove \cref{thm:triangle-removal-analysis},
note that if $i<\randT$ then by \cref{prop:H-count} the number of
triangles in $\randG\left(i\right)$ is
\[
\randQ\left(i\right)=\left(1\pm O\left(e\left(i\right)\right)\right)p^{3}\left(i\right)n^{3},
\]
which is positive as $e(i)\le\varepsilon^b$ and as long as we choose $\varepsilon'(a,h)$ small enough.

This means $\randT\le\randM$, so the event that each $\randY_{A}^{s}$
only takes negative values contains the event that each $\randG\left(i\right)$
is non-frozen and sufficiently quasirandom for $i\le N\left(1-\varepsilon^{b}\right)$.

Let $\randR_{A}\left(i\right)=\bigcap_{w\in A}N_{q}^{\left(i\right)}\left(w\right)$,
so that $\randY_{A}\left(i\right)=\left|\randR_{A}\left(i\right)\right|$.
Fix $A$, and consider $x\in\randR_{A}\left(i\right)$, for $i<\randT$.
The only way we can have $x\notin\randR_{A}\left(i+1\right)$ is if
we remove a triangle containing an edge $\left\{ x,w\right\} $ for
some $w\in A$. Now, for each $w\in A$, the number of triangles in
$\randG\left(i\right)$ containing the edge $\left\{ x,w\right\} $
is $\left(1\pm O\left(e\left(i\right)\right)\right)p^{2}\left(i\right)n$
by $(e(i),2)$-quasirandomness. The number of triangles containing $x$ and
more than one vertex of $A$ is $O\left(1\right)$. So, if $b$ is small enough then for realizations of $\randG\left(i\right)$ with $i<\randT$ we have
\begin{align*}
\Pr\left(x\notin\randR_{A}\left(i+1\right)\cond\randG\left(i\right)\right) & =\frac{1}{\randQ\left(i\right)}\left(\sum_{w\in A}\left(1\pm O\left(e\left(i\right)\right)\right)p^{2}\left(i\right)n-O\left(1\right)\right)\\
 & =\left(1\pm O\left(e\left(i\right)\right)\right)\frac{\left|A\right|}{p\left(i\right)N}.
\end{align*}
For $i<\randT$ we have $\left|\randR_{A}\left(i\right)\right|=\left(1\pm e\left(i\right)\right)p^{\left|A\right|}\left(i\right)n$,
so by linearity of expectation
\begin{align*}
\E\left[\Delta\randY_{A}\left(i\right)\cond\randG\left(i\right)\right] & =-\left(1\pm O\left(e\left(i\right)\right)\right)\frac{\left|A\right|p^{\left|A\right|-1}\left(i\right)n}{N}\\
 & =-\frac{\left|A\right|p^{\left|A\right|-1}\left(i\right)n}{N}+O\left(\frac{e\left(i\right)p^{\left|A\right|-1}\left(i\right)}{n}\right).
\end{align*}
Note also that we have the bound $|\Delta\randY_{A}\left(i\right)|\le2=O\left(1\right)$
(with probability 1). Also, for fixed $k$, we have
\begin{align*}
\Delta p^{k}\left(i\right) & =\left(1-\frac{i+1}{N}\right)^{k}-\left(1-\frac{i}{N}\right)^{k}\\
 & =\left(1-\frac{i}{N}\right)^{k}\left(\left(\frac{N-i-1}{N-i}\right)^{k}-1\right)\\
 & =p^{k}\left(i\right)\left(\left(1-\frac{1}{N-i}\right)^{k}-1\right)\\
 & =p^{k}\left(i\right)\left(-\frac{k}{N-i}+O\left(\frac{1}{\left(N-i\right)^{2}}\right)\right)\\
 & =-\frac{kp^{k-1}\left(i\right)}{N}\left(1+O\left(\frac{p\left(i\right)}{n^{2}}\right)\right)\\
 & =-\frac{kp^{k-1}\left(i\right)}{N}+o\left(\frac{e\left(i\right)p^{k-1}\left(i\right)}{n^{2}}\right),
\end{align*}
and with $ep^{k}$ denoting the pointwise product $i\mapsto e\left(i\right)p^{k}\left(i\right)$,
we then have
\begin{align*}
\Delta\left(ep^{k}\right)\left(i\right) & =\varepsilon^{c}\Delta p^{k-C}\left(i\right)\\
 & =\varepsilon^{c}\Theta\left(\frac{\left(C-k\right)p^{k-C-1}\left(i\right)}{N}\right)\\
 & =\Theta\left(\frac{\left(C-k\right)e\left(i\right)p^{k-1}\left(i\right)}{n^{2}}\right).
\end{align*}
Also, $\Delta(ep^k)(i) > 0$ for $C > k$.
For large $C$ it thus follows that
\begin{align*}
\E\left[\Delta\randY_{A}^{+}\left(i\right)\cond\randG\left(i\right)\right] & =\E\left[\Delta\randY_{A}\left(i\right)\cond\randG\left(i\right)\right]-\Delta p^{\left|A\right|}\left(i\right)n-\Delta\left(ep^{\left|A\right|}\right)\left(i\right)n\le0,
\end{align*}
and similarly
\[
\E\left[\Delta\randY_{A}^{-}\left(i\right)\cond\randG\left(i\right)\right]\le0
\]
for $i<\randT$. (For $i\ge\randT$ we trivially have $\Delta\randY_{A}^{s}\left(i\right)=0$)
Since each $\randY_{A}^{s}$ is a Markov process, it follows that
each is a supermartingale. Now, we need to bound $\Delta\randY_{A}^{s}\left(i\right)$
and $\E\left[\left(\Delta\randY_{A}^{s}\left(i\right)\right)^{2}\cond\randG\left(i\right)\right]$,
which is easy given the preceding calculations. First, recalling that
$\Delta\randY_{A}\left(i\right)=O\left(1\right)$ and noting that
$\Delta p^{k}\left(i\right),\Delta\left(ep^{k}\right)\left(i\right)=O\left(1/N\right)$
we immediately have $\left|\Delta\randY_{A}^{s}\left(i\right)\right|=O\left(1\right)$.
Noting in addition that $\E\left[\Delta\randY_{A}\left(i\right)\cond\randG\left(i\right)\right]=O\left(1/n\right)$,
we have
\begin{align*}
\E\left[\left(\Delta\randY_{A}^{s}\left(i\right)\right)^{2}\cond\randG\left(i\right)\right] & =O\left(\E\left[\Delta\randY_{A}^{s}\left(i\right)\cond\randG\left(i\right)\right]\right)=O\left(\frac{1}{n}\right).
\end{align*}
Since $\randT\le N$, we also have
\[
\sum_{i=m}^{\infty}\E\left[\left(\Delta\randY_{A}^{s}\left(i\right)\right)^{2}\cond\randG\left(i\right)\right]=O\left(\frac{N}{n}\right)=O\left(n\right).
\]
Provided $c<1$ and $C$ is large enough (and recalling that $\varepsilon<\varepsilon'(a,h)$), applying
\cref{lem:freedman} with $t=e\left(m\right)p^{\left|A\right|}\left(m\right)n-\varepsilon p^{\left|A\right|}\left(m\right)n=\Omega\left(n\varepsilon^{c}\right)$
and appropriate $v=O\left(n\right)$ then gives
\[
\Pr\left(\randY_{A}^{s}\left(i\right)>0\mbox{ for some }i\right)\le\exp\left(-O\left(n\varepsilon^{2c}\right)\right).
\]
Indeed, recall that we start at step $i=m$, and the initial quasirandomness conditions show that $\randY_A^s\left(m\right)\le \varepsilon p^{\left|A\right|}\left(m\right)n-e\left(m\right)p^{\left|A\right|}\left(m\right)n$.

So, if $2c<1/a\le\log_{1/\varepsilon}n$, the union bound over all $A,s$
finishes the proof.
\end{proof}

\bibliographystyle{amsplain_initials_nobysame_nomr}
\bibliography{main}

\end{document}